\documentclass{article}

\usepackage{times}
\usepackage{graphicx} 
\usepackage{subfigure}

\usepackage{algorithm}
\usepackage{algorithmic}

\usepackage{hyperref}

\usepackage[accepted]{icml2017}

\usepackage{amsmath,amssymb,amsthm,hyperref}
\usepackage{color}
\usepackage{xspace}

\usepackage[backend=bibtex,citestyle=authoryear,maxcitenames=2]{biblatex}
\newtheorem{lemma}{Lemma}

\newcommand{\norm}[1]{\left\Vert#1\right\Vert\xspace}

\newcommand{\expectn}[1]{\mathbb{E}\norm{#1}\xspace}

\icmltitlerunning{Distributed SAGA: Maintaining linear convergence rate with limited communication}
\begin{document}

\twocolumn[
\icmltitle{Distributed SAGA:\\Maintaining linear convergence rate with limited communication}



\icmlsetsymbol{equal}{*}

\begin{icmlauthorlist}
\icmlauthor{Cl\'ement Calauz\`enes}{crto}
\icmlauthor{Nicolas Le Roux}{crto}
\end{icmlauthorlist}

\icmlaffiliation{crto}{Criteo, Paris, France}
\icmlcorrespondingauthor{Cl\'ement Calauz\`enes}{c.calauzenes@criteo.com}
\icmlcorrespondingauthor{Nicolas Le Roux}{nicolas@le-roux.name}

\icmlkeywords{stochastic optimisation, large scale learning, distributed optimisation}

\vskip 0.3in
]



\printAffiliationsAndNotice{}  

\begin{abstract}
In recent years, variance-reducing stochastic methods have shown great practical performance, exhibiting linear convergence rate when other stochastic methods offered a sub-linear rate. However, as datasets grow ever bigger and clusters become widespread, the need for fast distribution methods is pressing. We propose here a distribution scheme for SAGA which maintains a linear convergence rate, even when communication between nodes is limited.
\end{abstract}

\section{Introduction}
Many problems in machine learning and optimisation can be cast as the minimisation of a finite but large sum of strongly convex functions, i.e.
\begin{align*}
f(w) &= \sum_{i=1}^N f_i(w) \; .
\end{align*}
In machine learning, for instance, each $f_i$ will be the loss for a particular datapoint and $N$ will be the size of the training set. Though the strong convexity might appear quite limiting, the use of an $\ell_2$ regulariser, quite common in machine learning, will satisfy this constraint

Methods to solve this problem have traditionally been split into two distinct classes: batch methods, dating back to~\textcite{cauchy1847methode}, which enjoy linear or super-linear convergence rates for strongly convex functions, and stochastic methods, stemming from the seminal work of~\textcite{robbins1951stochastic}, which enjoy sub-linear rates but with an iteration cost independent of the number of samples.

\textcite{leroux2012stochastic} introduced with SAG the first general stochastic optimisation method with a linear convergence rate for strongly convex functions, at the expense of a memory storage linear in the number of functions in the sum. This was quickly followed by other works in what are now called stochastic variance-reducing methods, such as  SDCA~\cite{Shalev-Shwartz2013sdca}, SVRG~\cite{johnson2013accelerating}, SAGA~\cite{defazio2014saga} or MISO~\cite{mairal2013optimization}. All these methods, however, assume that all the examples lie on the same machine and have immediate access to the vector of parameters. This has two consequences. First, as many of these methods have storage requirements linear in $N$, they are impractical for very large datasets. This is with the notable exception of SVRG which trades the memory requirement for a slower convergence rate (by a constant)\footnote{\textcite{babanezhad2015wasting} proposed a version of SVRG which would improve the constant factor without requiring more storage but the gains are in the early iterations.}. Second, they cannot take advantage of large clusters of machines. At a time where clusters of hundreds or even thousands of machines are commonplace, the gain in convergence speed obtained by these methods does not compensate for the speedup enjoyed by methods like L-BFGS (see, e.g.,~\cite{nocedal2006numerical}) which can easily distribute the gradient computations across the nodes of the cluster, the update step being performed on one machine.

The main reason stochastic methods are harder to distribute than batch methods such as L-BFGS is the cost of communication between nodes. Frequent communications between nodes increase the time per iteration as data transfers over a network are typically much slower than transfers between the RAM and the CPU. On the other hand, reducing the frequency of these communications increases the number of iterations required to reach the desired level of accuracy. Much of the research alleviated this problem by focusing on the multi-core setting, where communications are fast, and studying the impact of slower cores. Hogwild~\cite{recht2011hogwild} is a wildly successful implementation of a multi-core optimisation algorithm but does not address the issue of the data being spread across many nodes of the cluster.

Another popular approach to distributed stochastic methods is to use one or several parameter servers which are updated and called asynchronously by each of the workers. Downpour SGD~\cite{dean2012large} is one implementation and scales across thousands of machines. This method has been used to train large deep networks which give rise to non-convex loss where achieving linear convergence is less important since this convergence will only be towards a local optimum. However, it does not achieve a faster convergence rate when the loss function is convex.

The work closest to ours is the recent FADL~\cite{mahajan2013, mahajan2013b}, where the authors proposed a plugin method for distributing algorithms which achieved linear convergence and can be seen as a distributed variant of SVRG when used with the latter in inner loop. Their method required performing a line-search at each synchronisation and they do not provide insights on the difference between using batch and stochastic algorithms in inner loop, even though they explicitly use different hyper-parameter settings in their experiments. Our method differs in that we do not need to do a line search at every synchronisation, and we provide theoretical and practical understanding of the behaviour of a stochastic algorithm.

\section{Distributed SAGA}
For the remainder of this paper, we assume that we have $K$ machines, each having access to a subset of the datapoints. In particular, each machine will be able to store both its subset of the datapoints and their associated gradients~\footnote{Or, only a sufficient statistics, as detailed in the SAG paper.} in memory. Finally, communication costs between machines will be assumed high and the major bottleneck in optimisation~\footnote{A model with $10^{12}$ parameters, for instance, will use 4GB of storage and might take between 10s and 30s to be averaged between a set of machines.}.

In SAGA, one of the methods mentioned in the introduction, the last gradient computed for each datapoint is kept in memory. The important aspect, and the main different with a standard batch method, is that these gradients have been computed at different times using different values of the parameter vector. Since linear convergence is achieved despite this staleness, it seems natural to want to include stale gradients from the other nodes in the parameter updates. One must still be careful as, contrary to its original version, distributed SAGA, or dSAGA for short, cannot update all the gradients. We shall show how a simple modification can maintain linear convergence despite this hurdle.

Every once in a while, each node will communicate to the others its current position in parameter space as well as the sum of its stored gradients. The storage requirement of these two vectors is independent of the dataset size, which makes it efficient both in terms of communication and in terms of memory. This makes it possible for dSAGA to scale to datasets of any size, provided that the number of nodes grows linearly with the number of datapoints~\footnote{Increasing the number of nodes might however decrease the convergence rate.}.

This approach is very close to that of FADL~\cite{mahajan2013} which builds at each synchronisation $t$ and on each worker $k$ an approximation $f^t_k$ of the global objective function $f$. The biggest difference is that, while FADL performs a line-search at each synchronisation to find a new set of parameters, dSAGA does so by a simple average of the local parameters at each node. Further,~\textcite{mahajan2013} prove the linear convergence of FADL, albeit in a generic case w.r.t. the algorithm in inner loop. Thus, it provides a slower rate than gradient descent and no insight on how the convergence speed varies with the number of workers and the synchronisation frequency. In contrast, we focus on a specific algorithm for the inner loop, SAGA, in order to provide better practical insights on the variation of the convergence speed depending on the setting.

Besides its speed, one of the main selling points of SAGA is its total absence of hyper-parameters. This makes it well suited to a production environment where robustness is required. Since distributed algorithms structurally have more parameters than their non-distributed counterpart, such as the number of nodes and the synchronisation frequency, we provide an understanding of the variation of the speed with the hyper-parameters to limit the need for tuning.

\subsection{Algorithm}
We now describe dSAGA in terms of the parameter updates it performs rather than in terms of the loss it optimises as we believe it makes the approximations clearer. We will start with a reminder of the original SAGA algorithm to highlight the differences.

At time step $t$, SAGA does the following update:
\begin{algorithm}
\caption{Update for SAGA}
\label{alg:saga}
\begin{algorithmic}
\INPUT A current parameter vector $w^t$
\INPUT Stored gradient $g_i(\phi_i^t)$ for each $i$
\begin{enumerate}
\item Choose a $j$ at random.
\item Set $\phi_j^{t+1} = w^t$ and store $g_j(\phi_j^{t+1})$.
\item Leave $\phi$ unchanged for all the other points, i.e. $\phi_i^{t+1} = \phi_i^t$ for $i \neq j$.
\item Update $w$ using \\$w^{t+1}  = w^t {\rm -} \gamma \left[g_j(\phi_j^{t+1}) - g_j(\phi_j^t) + \frac{1}{N}\sum_i g_i(\phi_i^t)\right] \; $.
\end{enumerate}
\end{algorithmic}
\end{algorithm}

As we see, SAGA updates only one gradient at a time, using stale versions of all the other gradients. A point worth noting is that the algorithm also works if only the last term, the sum over all datapoints, is used to compute the update. Although no convergence rate has been proven for this slight variation~\footnote{This variation is very close to SAG, with the slight modification that the entry for point $j$ is updated after the parameter update and not before.}, it hints at the fact that $\frac{1}{N}\sum_i g_i(\phi_i^t)$ is a reasonable approximation of the current gradient.

Before moving to the description of dSAGA, we need to set a few notations:
\begin{itemize}
\item Each node has its own parameter vector $w_k$. Since we will run a new optimisation between each synchronisation between nodes, these parameters will have two indices. We thus use $w_k^{t, u}$ where $t$ is the index of the synchronisation and $u$ is the iteration since the last synchronisation.
\item The $i$-th point on node $k$ will be represented by the pair $(k, i)$. $f_{k,i}$ is its loss with associated gradient $g_{k, i}$. Thus the gradient on the $k$-th node is $g_k = \frac{1}{K}\sum_i g_{k,i}$.
\item $\hat{g}_k(w_k^{t,u})$ is the approximation of the average gradient on node $k$ at $w_k^{t,u}$. In other words, when node $k$ reaches $w_k^{t,u}$, we look at all the stored $g_{k,i}(\phi_i^{t,u})$ for the points on node $k$ and set $\hat{g}_k(w_k^{t,u}) = \frac{1}{N} \sum_i g_{k,i}(\phi_i^{t,u})$.
\item Each node will perform $U$ passes through the data before synchronising with the other nodes. Since there are $N/K$ datapoints on each node, this means that $U = uN/K$.
\end{itemize}

Rewriting the SAGA update on node $k$ in this context yields:
\begin{enumerate}
\item Choose a $j$ at random.
\item Set $\phi_{k,j}^{t, u+1} = w_k^{t, u}$ and store $g_{k,j}(\phi_{k,j}^{t, u+1})$.
\item Leave $\phi$ unchanged for all the other points, i.e. $\phi_{k,i}^{t, u+1} = \phi_{k,i}^{t, u}$ for $i \neq j$.
\item Update $w_k$ using $w_k^{t, u+1}  = w_k^{t,u} - \gamma \delta_k^{t,u}$ with
\begin{align*}
\delta_k^{t,u} &= g_{k,j}(\phi_{k,j}^{t, u+1}) - g_{k,j}(\phi_{k,j}^{t,u}) + \frac{1}{N_k}\sum_i g_{k,i}(\phi_{k,i}^{t,u})\\
&= g_{k,j}(\phi_{k,j}^{t, u+1}) - g_{k,j}(\phi_{k,j}^{t,u}) + \hat{g}_k(w_k^{t,u})\; .
\end{align*}
\end{enumerate}

One could think that all there is to do to distribute SAGA  is to add to the update the sum of the stored gradients which was communicated at the last synchronisation. However, when updating $w_k^{t,u}$ on node $k$, we would like to have an estimate of the gradient on the other nodes at the same point, i.e. $\widehat{g}_l(w_k^{t,u})$. What we have access to, however, is the estimate of the gradient on these nodes at the time of the synchronisation, i.e. $\widehat{g}_l(w_l^{t-1,+\infty})$. We make here the assumption that each optimisation is run to convergence before synchronisation. We thus need to obtain an estimate of  $g_l(w_k^{t,u})$. Assuming our ideal SAGA update would be $w_k^{t, u+1} = w_k^{t,u} - \gamma \delta_k^{t,u}$ with
\begin{align*}
\delta_k^{t,u}&= g_{k,j}(\phi_{k,j}^{t, u+1}) - g_{k,j}(\phi_{k,j}^{t,u}) + \sum_l g_l(w_k^{t,u}) \; ,
\end{align*}
we make the following approximation of $\sum_l g_l(w_k^{t,u})$:
\begin{align*}
\sum_l g_l(w_k^{t,u}) &= \sum_l \left[g_l(w_k^{t,u}) - g_l(w_k^{t,0})\right.\\
 &\quad + \left. g_l(w_k^{t,0}) - g_l(w_l^{t-1,+\infty}) + g_l(w_l^{t-1,+\infty})\right]\\
&\approx \sum_l \left[g_k(w_k^{t,u}) - g_k(w_k^{t,0}) \right.\\
&\quad \left.+ g_k(w_k^{t,0}) - g_k(w_l^{t-1,+\infty}) + g_l(w_l^{t-1,+\infty})\right]
\end{align*}
where we assumed that the Hessians on each node were the same, leading to $g_l(w) - g_l(v) = g_k(w) - g_k(v)$ for any $(w, v)$ pair.

$w_k^{t,0}$ is the initial point of the optimisation of node $k$ which we can choose as we please. Choosing $\displaystyle w_k^{t, 0} = \frac{1}{K} \sum_l w_l^{t-1, +\infty}$ and assuming the loss is locally quadratic, we have $\displaystyle \sum_l\left[g_k(w_k^{t,0}) - g_k(w_l^{t-1,+\infty})\right] = 0$ and then
\begin{align*}
\sum_l g_l(w_k^{t,u}) &\approx K \left[g_k(w_k^{t,u}) - g_k(w_k^{t,0})\right] + \sum_l g_l(w_l^{t-1,+\infty}) \; .
\end{align*}
Replacing $g_k(w_k^{t,u})$ with the unbiased estimate used in SAGA, this leads to the dSAGA update describe in Alg.~\ref{alg:dsaga}.
\begin{algorithm}
\caption{dSAGA update on node $k$}
\label{alg:dsaga}
\begin{algorithmic}
\INPUT{A current parameter vector $w_k^{t, u}$ }
\INPUT{Stored gradient $g_{k, i}(\phi_{k, i}^{t, u})$ for each $i$}
\INPUT{The average of all gradient estimates at the last synchronisation $\frac{1}{K}\sum_l g_l(w_l^{t-1,\infty})$}\\
Recompute $g_k(w_k^{t,0}$. This is the \emph{local gradient pass}.\\
For $u = 1$ to $UN/K$, do:
\begin{enumerate}
\item Choose a $j$ at random.
\item Set $\phi_{k,j}^{t, u+1} = w_k^{t, u}$ and store $g_{k,j}(\phi_{k,j}^{t, u+1})$.
\item Leave $\phi$ unchanged for all the other points, i.e. $\phi_{k,i}^{t, u+1} = \phi_{k,i}^{t, u}$ for $i \neq j$.
\item Update $w_k$ using
\begin{align*}
w_k^{t, u+1}  &= w_k^{t,u} + \gamma  g_k(w_k^{t,0}) - \gamma \frac{1}{K}\sum_l g_l(w_l^{t-1,\infty}) \\
&\; - \gamma \left[g_{k,j}(\phi_{k,j}^{t, u+1}) - g_{k,j}(\phi_{k,j}^{t,u}) + \hat{g}_k(w_k^{t,u})\right] \; .
\end{align*}
\end{enumerate}
\end{algorithmic}
\end{algorithm}

We now study the convergence speed of dSAGA, both theoretically and empirically.
\section{Convergence Proof in the Quadratic Case}
\label{sec:theory}

Let us assume we wish to minimise
\begin{align*}
f(w) &= \frac{1}{K} \sum_{k=1}^K f_k(w)\\
     &= \frac{1}{K} \sum_{k=1}^K \frac{K}{N}\sum_{j = 1}^{N/K} f_{k,j}(w) \; ,
\end{align*}

We are interested in upper bounding the rate at which the expected excess error $\expectn{w^{t,0} - w^*}$ decreases from one synchronisation to another where the expectation is taken over the sampled sequences of examples in inner loops.

We have the following decomposition, that we will also study in the experimental part (Sec. \ref{sec:xp}).
\begin{align}
	\hspace{1em}&\hspace{-1em}\expectn{w^{t+1,0} - w^*} \nonumber\\
	& \leq \frac{1}{K}\sum_{k} \expectn{w_k^{t,u} - w^*} \nonumber\\
	& \leq \underbrace{\frac{1}{K}\sum_{k} \expectn{w_k^{t,u} - w_k^{t,\infty}}}_{\text{average inner error}} + \underbrace{\frac{1}{K}\sum_{k} \expectn{w_k^{t,\infty} - w^*}}_{\text{discrepancy error}}
	 \label{eq:regret_decomposition}
\end{align}
The first term, the \emph{average inner error} represents how much we can hope to reduce the error without further synchronisation. It decreases at the linear speed of the inner SAGA algorithm. The second, the \emph{discrepancy error} represents the incompressible error due to the lack of information on each node from the other nodes, that we cannot reduce without further synchronisation.

In this paper, we prove the convergence of the dSAGA algorithm when each inner loop is run to convergence, i.e. for $u \to +\infty$. Thus, we focus on the convergence of the discrepancy error.

\begin{lemma}[Convergence]
\label{lem:convergence}
We assume that, for each $k$, we have
\begin{align*}
\frac{K}{N} \sum_{j = 1}^{N/K} f_{jk}(w) &= \frac{1}{2} (w - w_k^*)^\top H_k (w - w_k^*) \; ,
\end{align*}
each $H_k$ being positive definite and that for any $t$,
$$w^{t+1,0} = \frac{1}{K}\sum_k w_k^{t,\infty}\,.$$

If $\max_{k,l}\|I - H_k^{-1} H_l\| \leq \rho < 1$, we have for each $l$
\begin{align*}
\|w_l^{t, 0} - w^*\|   & \leq \left(1 - \frac{1}{K}\right)^t\rho^t C_0 \; ,
\end{align*}
where $w^*$ is the global minimising of $f$ and $C_0$ a constant depending on the starting point.
\end{lemma}

\begin{proof}
Using the update equation from dSAGA, running the local optimisation to convergence yields
\begin{align*}
g_k(w_k^{t, \infty}) = & g_k(w_k^{t,0}) - \frac{1}{K}\sum_l g_l(w_l^{t-1, \infty})
\end{align*}
with $g_k$ the derivative of $f_k$.
Since $f_k$ is quadratic, we have $g_k(w) = H_k(w - w_k^*)$ and the previous equation may be rewritten (multiplying by $H_k^{-1}$ and adding $w_k^*$)
\begin{align*}
w_k^{t, \infty} &= w_k^{t, 0} - \frac{1}{K}\sum_l H_k^{-1} H_l(w_l^{t-1, \infty} - w_l^*) \; .
\end{align*}
and thus, subtracting $w^*$ on both sides of the equation yields
\begin{align*}
w_k^{t, \infty} - w^*&= w_k^{t, 0} - w^* - \frac{1}{K}\sum_l H_k^{-1} H_l(w_l^{t-1, \infty} - w_l^*) \; .
\end{align*}
Using the definition of $f$, we have
\begin{align*}
\sum_l H_l w^* &= \sum_l H_l w_l^*
\end{align*}
and thus
\begin{align*}
w_k^{t, \infty} - w^*&= w_k^{t, 0} - w^* - \frac{1}{K}\sum_l H_k^{-1} H_l(w_l^{t-1, \infty} - w^*) \; .
\end{align*}
Further, since we chose $w_k^{t, 0} = \frac{1}{K} \sum_l w_l^{t-1, \infty}$, we have
\begin{align*}
w_k^{t, 0} - w^* &= \frac{1}{K} \sum_l w_l^{t-1, \infty} - w^*\\
&= \frac{1}{K} \sum_l (w_l^{t-1, \infty} - w^*) \; ,
\end{align*}
and
\begin{align*}
w_k^{t, \infty} - w^*&= \frac{1}{K}\sum_l (I - H_k^{-1} H_l)(w_l^{t-1, \infty} - w^*)\\
&= \frac{1}{K}\sum_{l\neq k} (I - H_k^{-1} H_l)(w_l^{t-1, \infty} - w^*) \; .
\end{align*}
We may now upper bound $\|w_k^{t, \infty} - w^*\|$ using
\begin{align*}
\hspace{1em}&\hspace{-1em}\|w_k^{t, \infty} - w^*\| \\
&\leq \frac{1}{K}\sum_{l\neq k} \|I - H_k^{-1} H_l\|\|w_l^{t-1, \infty} - w^*\|\\
&\leq \frac{K-1}{K} \max_{l} \|I - H_k^{-1} H_l\|\|w_l^{t-1, \infty} - w^*\|\\
&\leq \frac{K-1}{K} \max_{l} \|I - H_k^{-1} H_l\|\max_l \|w_l^{t-1, \infty} - w^*\|\; .
\end{align*}
If we assume that there is a $\rho$ such that
\begin{align*}
\max_{l} \|I - H_k^{-1} H_l\| &\leq \rho \; ,
\end{align*}
then taking the maximum over $k$ yields
\begin{align*}
\max_k \|w_k^{t, \infty} - w^*\| &\leq \frac{K-1}{K} \rho \max_l\|w_l^{t-1, \infty} - w^*\| \; .
\end{align*}
This concludes the proof.
\end{proof}

Now that we have an asymptotic guarantee of convergence, we would like to be able to estimate the convergence ratio $\rho$, at least in a very simple setting.

\begin{lemma}[Value of $\rho$]
\label{lem:rho}
We assume to minimise a linear model under gaussian data distributed over $K$ nodes. More precisely, $f_k(w) = (w-w_k^*)^T H_k (w - w_k^*)$ with
\begin{align*}
H_k &= \frac{K}{N} \sum_{i=1}^{N/K} x_{k,i}x_{k,i}^T
\end{align*}
where the datapoints $x$ have been drawn from a Gaussian with mean $0$ and covariance $\Sigma$ in dimension $d$,
\begin{align*}
x &= \Sigma^{1/2} u ~~~~~\text{with}~ u \sim \mathcal{N}(0, I) \; .
\end{align*}

Then, for $N\to+\infty$, $d\to +\infty$ and $\frac{dK}{N} \to \gamma$, we have
\begin{align*}
\|I - H_i^{-1}H_j\| &\to \frac{2\sqrt{\gamma}}{1 - \sqrt{\gamma}} \; .
\end{align*}
\end{lemma}

\begin{proof}
Denoting $r_{ij} = \|I - H_i^{-1} H_j\|$, we see that $r_{ij}$ is small when $H_i$ and $H_j$ are similar. This makes sense as, if the problems on different nodes are similar, our approximation will be good and the convergence will be fast.

From our assumptions, we can write
\begin{align*}
H_i &= \frac{K}{N} \sum_{i} \Sigma^{1/2} u_{k,i} u_{k,i}^T\Sigma^{1/2}\; .
\end{align*}
where $\hat{\mu}_k = \frac{K}{N} \sum_{i} u_{k,i}$ is the empirical average on node $k$.

We know that $\frac{K}{N} \sum_{i}u_{k,i} u_{k,i}^T$ follows a Wishart distribution $\mathcal{W}_d(I, \frac{N}{K})$. In the limit of $d \rightarrow \infty$, $N \rightarrow \infty$, $d/N \rightarrow \gamma$, we have the following results:
\begin{itemize}
\item The eigenvalues of a Wishart matrix are uniformly distributed (by symmetry)
\item The norm of a Wishart matrix drawn from $\mathcal{W}_d(I, \frac{N}{K})$ converges almost surely to $(1 + \sqrt{\gamma})^2$ \cite{geman1980limit}
\item The empirical eigenvalue distribution converges almost surely to $$\displaystyle\mu_{\gamma}(x) = \frac{\sqrt{4\gamma - (x-1-\gamma)^2}}{2\pi\gamma x}\chi(x)$$ where $\chi$ denotes the characteristic function of the interval $[(1 - \sqrt{\gamma})^2, (1 + \sqrt{\gamma})^2]$ \cite{jonsson1982some}.
\end{itemize}

Since the eigenvectors of both $M$ and $N$ are uniformly distributed over the unit sphere (by symmetry), we have that the expected norm of $M^{-1}N$ is $\frac{Tr(M^{-1}) \|N\|}{d}$.

Knowing the empirical eigenvalue distribution, we can compute the limit of $Tr(H_i^{-1})$ using
\begin{align*}
\lim_{N/K \rightarrow \infty} Tr(H_i^{-1}) &= d\int_{(1 - \sqrt{\gamma})^2}^{(1 + \sqrt{\gamma})^2}\frac{\sqrt{4\gamma - (x-1-\gamma)^2}}{2\pi\gamma x^2} \; dx\\
    &= \frac{d}{1 - \gamma} \; .
\end{align*}
Thus, we have $\displaystyle \|H_i^{-1}H_j\| = \frac{(1 + \sqrt{\gamma})^2}{1 - \gamma}$
and
\begin{align*}
\|I - H_i^{-1}H_j\| &= \frac{(1 + \sqrt{\gamma})^2}{1 - \gamma} - 1\\
    &= \frac{2\sqrt{\gamma}}{1 - \sqrt{\gamma}} \; .
\end{align*}
This concludes the proof.
\end{proof}
We emphasize that this result is based on several assumptions not verified in practice and should only serve as a rough guide to the actual convergence rate.

\section{Real-world performance of dSAGA}
\label{sec:xp}
We now study the empirical performance of dSAGA on several datasets for various settings of the hyper-parameters $U$ and $K$. First, we want to know how well dSAGA performs compared to L-BFGS which is the most popular algorithm for distributed convex optimisation when reaching high precision is important~\cite{agarwal14a}. Second, we want to know how rarely we can communicate between nodes while maintaining an acceptable convergence speed. Finally, we explore how dSAGA scales with the number of nodes $K$.

\subsection{Details of the experiments}
We chose three datasets to study the behaviour of dSAGA:
\begin{itemize}
\item ALPHA\footnote{ftp://largescale.ml.tu-berlin.de/largescale/alpha} is a dataset from the Large-Scale Pascal Challenge containing $N=5e5$ samples in dimension $d=500$. Even though the size of this dataset is too small for distributed learning to be meaningful, its small dimension makes the comparison between theoretical and empirical values of constants such as $\rho$ feasible.
\item CriteoSmall\footnote{http://labs.criteo.com/downloads/2014-kaggle-display-advertising-challenge-dataset} is a dataset released by Criteo in the context of a Kaggle challenge. It contains $N=4.5e7$ binary samples in dimension $N=1e6$, each sample having approximately $100$ nonzero components.
\item CriteoLarge\footnote{http://labs.criteo.com/2013/12/download-terabyte-click-logs} is one of the largest public datasets for classification and contains $N=4e9$ samples in dimension $d=1e6$. The sparsity is the same as for CriteoSmall.
\end{itemize}

We minimised an $\ell_2$-regularised logistic regression. For dSAGA, we set the stepsize of the inner loop at the theoretical value provided for SAGA~\cite{defazio2014saga}. Similar to what has been done by~\textcite{agarwal14a} and~\cite{defazio2014saga}, we initialised both L-BFGS and dSAGA with one pass of stochastic gradient in all our experiments. The convergence speed was measured in log-excess error as a function of number of passes through the data. We decided against using wall clock time as this was heavily dependent on the specific implementation of the distributed system, Spark in our case, and clouded actual differences between optimisation algorithms.

Finally, dSAGA requires a local gradient pass after each synchronisation, which costs an extra pass over the data. Unless otherwise mentioned, this was accounted for in all the figures.

\subsection{Communication at every pass}
\label{sec:every_pass}
We first compare dSAGA and L-BFGS when synchronisation between nodes happens after only one optimisation pass through the data, i.e. $U=1$. This setting is the most detrimental to dSAGA as, due to local gradient pass, only half of the passes through the data are effectively used for optimisation. Fig.~\ref{fig:one_pass} shows that dSAGA outperforms L-BFGS on ALPHA, performs comparably on CriteoSmall but is slower than L-BFGS on CriteoLarge.

Setting $U=1$, however, is only useful if the local convergence on each node, or inner optimisation, is much faster than the global convergence, or outer optimisation. If each local optimisation is slow, then it should be beneficial to reduce the synchronisation frequency, i.e. increase $U$, to improve the overall speed. Section \ref{sec:comm_freq} explores the impact of $U$ on the convergence speed.

\begin{figure*}[h]
\includegraphics[width=0.32\textwidth]{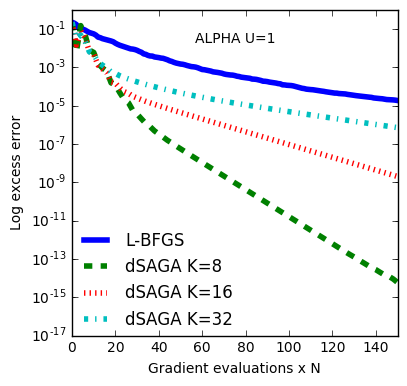}
\includegraphics[width=0.32\textwidth]{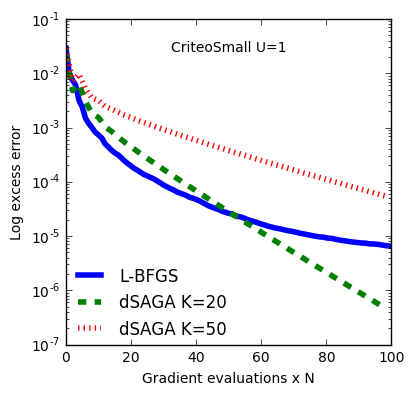}
\includegraphics[width=0.32\textwidth]{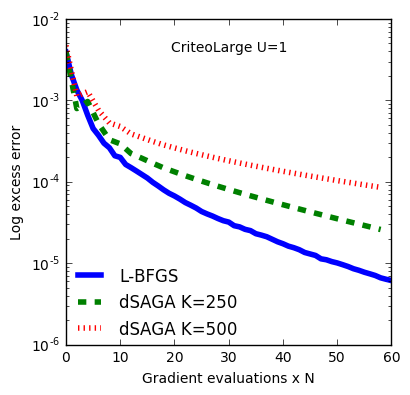}
\caption{Comparison of convergence speed when synchronising every pass on the data.}
\label{fig:one_pass}
\end{figure*}


\subsection{dSAGA in the large cluster regime}
\label{sec:increasing_K}
As the storage cost of dSAGA scales linearly with the number of samples on each node, maintaining a constant storage cost per node requires the number of nodes $K$ to scale linearly with the dataset size. For large datasets, $K$ can be in the hundreds to maintain acceptable local storage costs.

Increasing $K$ has two consequences. First, the amount of communication increases as each node must receive an update from every other node. Depending on the underlying implementation (e.g., tree aggregation or flooding), the complexity of such a communication varies between ${\cal O}\left(\log(K)\right)$ and ${\cal O}(K)$. Further, as the number of different local optimisation problems increases, so is the probability of a high discrepancy between several local optima, thus slowing down the optimisation.

The first issue is inherent to any distributed algorithm and can be tackled by reducing the communication frequency, which will be covered in the next section. We focus now on the second issue, studying how increasing $K$ affects the convergence speed. Section~\ref{sec:theory}, and in particular Eq.~\ref{eq:regret_decomposition}, show that two components adversely impact that speed.

First, each local hessian $H_k$ gets further apart from the average Hessian $H$, increasing the discrepancy between local optimisations and limiting the improvement one can obtain without synchronisation. Second, SAGA's convergence speed is dependent on the number of datapoints, equal to $N/K$, and local optimisation slows down when $K$ is too large. This degradation is especially important when $N/K$ gets close to the condition number $\kappa=\frac{L}{\mu}$.

Fig.~\ref{fig:one_pass} shows how the convergence speed of dSAGA degrades when $K$ increases. The degradation occurs on all datasets and does so much earlier when the number of samples is $N$ is small. Thus, our recommendation is to set $K$ to the minimum value allowing each local node to store its samples and gradients.


\subsection{Reducing the communication frequency}
\label{sec:comm_freq}
When reducing the frequency of communication, each inner optimisation runs closer to convergence. In the extreme limit of $U=+\infty$, each local optimisation will have converged to a different optimum and no communication will have occurred, leading to a poor global solution. The goal is thus to set $U$ large enough so that each local optimisation gets close to convergence, but not too large so that local parameter vectors do not stray too far from each other.

Fig.~\ref{fig:latency_variations} shows the impact of $U$ on the global convergence speed. As increasing $U$ compounds two effects, the beneficial one of reducing ``wasted'' local gradient passes post-synchronisation over the data and the detrimental one of potentially over-optimising local functions, we split the results in two columns. The left column only counts optimisation passes through the data. With this metric, increasing $U$ can only deteriorate the convergence speed. The right column takes every pass into account, thus providing a comparison closer to actual performance.

We first notice, that except with $U=12$ for ALPHA, increasing $U$ does not degrade the performance in terms of optimisation passes. This means, for all these datasets, the limiting factor is the local convergence. This is good news as this means that we can increase $U$, thus reducing the communication overhead, without loss of performance. This is confirmed with the plots on the right column when larger values of $U$ consistently outperform smaller ones.

With reduced communication frequency, dSAGA now clearly outperforms L-BFGS on CriteoSmall and is comparable on CriteoLarge. Unfortunately, computations were too intensive to run exhaustive experiments on the impact of $U$ on the larger datasets.

These results means that the extra-pass on the data required after synchronisation in dSAGA as well as the communication cost can be absorbed by increasing $U$. In the end, the convergence speed of dSAGA is quite robust to limited communications, which makes the algorithm competitive when communication cost is high.

\begin{figure}[h]
\includegraphics[width=0.5\textwidth]{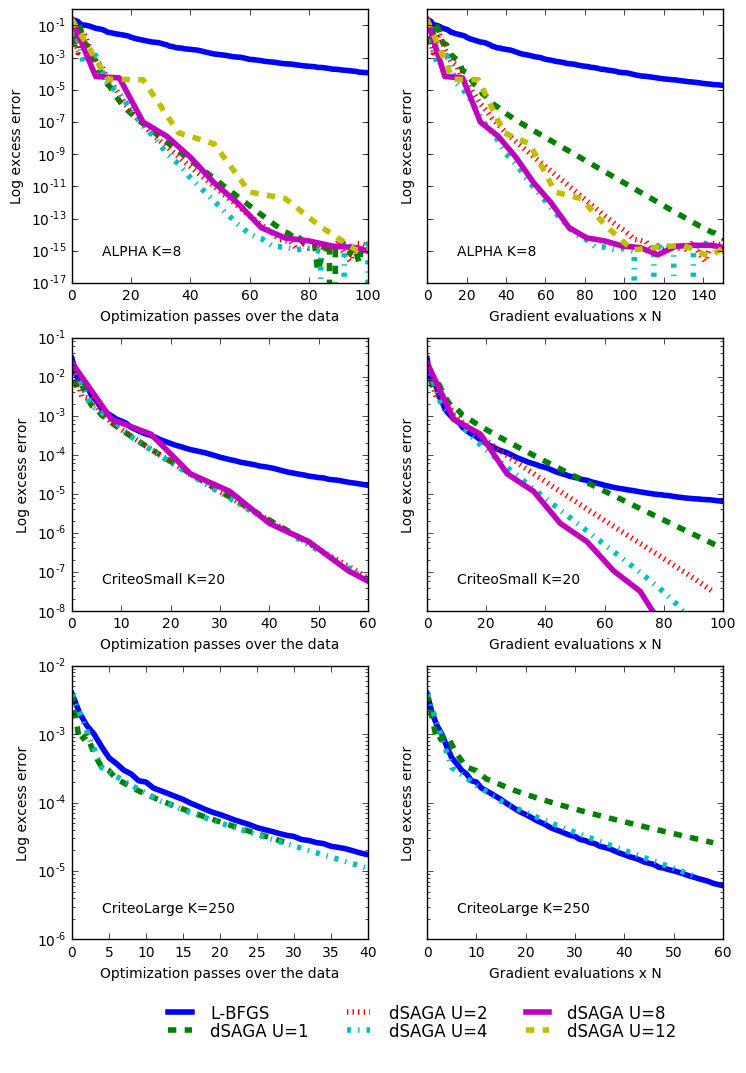}
\caption{Impact of $U$ on the convergence speed. The left column only counts optimisation passes while the right column takes all passes into account, thus being more representative of real performance. dSAGA is robust to decreased communications and is competitive or beats L-BFGS on all datasets.}
\label{fig:latency_variations}
\end{figure}

\subsection{Analysis of the inner and outer convergence speed}
The experiments shown in Fig.~\ref{fig:latency_variations} showed we can increase the latency without impacting the convergence speed, which indicates that the \emph{inner error} term dominates the \emph{discrepancy error} term. Thus, it is likely that the degradation with $K$ we observe in Fig.~\ref{fig:one_pass} is due to the degradation of the convergence speed of the inner SAGA algorithm.

To observe this more precisely, we let the inner loops run to convergence by setting $U=100$ to observe the empirical convergence rates of the \emph{inner error} and of the \emph{discrepancy error}.
We denote $\tilde\rho_K$ the reduction in excess error of $f$ between two consecutive synchronisations, i.e.
\begin{align}
\tilde\rho_K &= \frac{f(w_k^{t+1,0}) - f^*}{f(w_k^{t,0}) - f^*} \; .
\end{align}
We also denote
\begin{align}
\hat\rho_K &= \left(1-\frac{1}{K}\right)\frac{2\sqrt{\gamma}}{1 - \sqrt{\gamma}}
\end{align}
the theoretical valued predicted by Lemma~\ref{lem:rho}.
Lemma~\ref{lem:convergence} tells us that, if the local optimisations did indeed converge, we should have $\hat\rho_K = \tilde\rho_K$.
Then, we denote $\tilde\alpha_K$ the reduction in the excess error of $f$ during the inner optimisation after one pass over the data, taking the worst one over all the nodes:
\begin{equation}
\tilde\alpha_K = \max_k \frac{f\left(w_k^{t,u+\frac{N}{K}}\right) - f\left(w_k^{t,\infty}\right)}{f(w_k^{t,u}) - f\left(w_k^{t,\infty}\right)} \; .
\end{equation}
Finally, we denote $\tilde\omega_K$ the empirical convergence speed of SAGA on the inner function,
\begin{equation}
\tilde\omega_K = \max_k \frac{f_k^t\left(w_k^{t,u+\frac{N}{K}}\right) -  f_k^t\left(w_k^{t,\infty}\right)}{f_k^t(w_k^{t,u}) - f_k^t\left(w_k^{t,\infty}\right)} \; .
\end{equation}
where $$f_k^t(w) = f_k(w) + \frac{1}{K}\!\!\sum_l \left(\hat g_l(w_l^{t\text{-}1,\infty}) - g_k(w^{t,0})\right)^{\!\!T}\!\!(w - w^{t,0})$$

Fig.~\ref{fig:geometric_rates_k} shows how these four quantities vary as a function of the number of nodes $K$ on the ALPHA dataset.

\begin{figure}[h]
\includegraphics[width=0.5\textwidth]{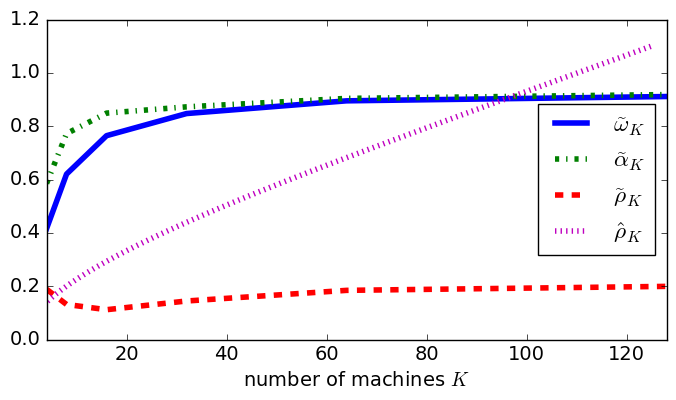}
\caption{Variation of the convergence speed of inner improvement versus best guaranteed regret when $K$ varies on ALPHA.}
\label{fig:geometric_rates_k}
\vspace*{-.3cm}
\end{figure}

Let us first focus on the excess error of $f$ after each synchronisation. Save for the lowest values of $K$, we observe that the decrease in excess error observed empirically, $\tilde\rho_K$ remains almost constant. This indicates that the decrease in convergence speed observed when increasing $K$ is almost exclusively due to the inner optimisation of the $f_k$'s slowing down.

This is confirmed by the variation of $\tilde\omega_K$. For low values of $K$, one pass through the data leads to a decrease in local excess error of more than half. However, it quickly degrades until, for large values of $K$, the decrease is worse than 0.8.

The degradation of $\tilde\omega_K$, the local convergence speed, drives the degradation of $\tilde\alpha_K$, the reduction in the excess error of $f$ during the inner optimisation.

Finally, we can observe that the theoretical bound on $\tilde\rho_K$, $\hat\rho_K$ is accurate for low values of $K$ but quickly becomes very loose. This is due to the proof technique used where an average over all nodes has been replaced with a maximum over those nodes. As the number of nodes increases, this bound becomes less accurate and the theoretical guarantees are far worse than the empirical results.

%

This leads us to conclude that, in all the regimes we studied, the convergence speed of dSAGA is mainly driven by the convergence speed of the local optimisations. This is promising as it shows the robustness of our distribution technique and its potential applicability to faster local optimisers than SAGA.

\section{Conclusion}
We proposed a distribution version of the SAGA algorithm which, by communicating an approximation of the state of each node to all the other nodes, outperforms L-BFGS while limiting the amount of communication between nodes. In settings where achieving high accuracy with large datasets is important, such as in online advertising, dSAGA offers the benefits of SAGA while maintaining a constant storage cost per node.

We also showed how dSAGA was affected by an increase in the number of nodes, concluding that, in order to maintain good speeds, each node should contain enough samples.

A few directions remain to be explored. First, it would be interesting to apply dSAGA to more recent optimisation techniques such as the catalyst method of~\textcite{lin2015universal}. We also think that designing an algorithm suited to asynchronous communication would provide great benefit, especially in the large $K$ setting where the likelihood of a node much slower than the others is greater.

\subsubsection*{Acknowledgements}
The authors would like to thank Victor Manuel Guerra Moran for his help in implementing dSAGA.
\printbibliography

\end{document}